\documentclass[12pt]{extarticle}
\usepackage{amsfonts,amsmath}
\usepackage{amssymb}
\usepackage{latexsym}
\usepackage{epsfig}
\usepackage{graphicx,color,graphics}

\usepackage{caption}
\usepackage{subcaption}
\usepackage{xcolor}
\usepackage{tcolorbox,color}
\usepackage{lineno}
\usepackage{pgf,tikz}
\usepackage{mathrsfs}
\usepackage{float}
\usepackage{multicol}
\usepackage{enumitem}
\usepackage{mdwlist}
\usepackage{dsfont}
\usetikzlibrary{arrows}
\usetikzlibrary{patterns}
\newtheorem{lemma}{Lemma}[section]

\newtheorem{theorem}{Theorem}[section]

\newtheorem{remark}{Remark}[section]

\newcommand{\N}{\ifmmode{{\Bbb N}}\else{\mbox{${\Bbb N}$}}\fi}
\newcommand{\R}{\ifmmode{{\Bbb R}}\else{\mbox{${\Bbb R}$}}\fi}

\begin{document}
\title{Absence of Exponential Stability and Polynomial Stabilization in a Class of Beam Models with Tip Rotary Inertia}

\author{Gerardo G\'{o}mez \'{A}valos$^{1,}$\footnote{gegova@gmail.com, gerardo.gomez@unab.cl},\ \ Jaime Mu\~noz Rivera$^{2,}$\footnote{jemunozrivera@gmail.com}, \ \  Elena Ochoa Ochoa$^{3,}$ \footnote{elenaochoaochoa18@gmail.com}.  \\  \small \it $^{1}$ Universidad Andres Bello, Departamento de Matem\'aticas, Facultad de Ciencias Exactas,\\
	\small \it Sede Concepci\'on, Autopista Concepci\'on-Talcahuano 7100, Talcahuano, Chile.\\
	\small \it $^{2,3}$Universidad del B\'io B\'io, Departamento de Matem\'aticas, Facultad de Ciencias,\\
	\small \it Avenida Collao 1202, Concepci\'on, Chile.\\
	\small \it $^{2}$ Laborat\'orio Nacional de Computaci\'on Cient\'ifica, Petr\'opolis - Brazil.				    
}	
\date{}
		\maketitle
	\begin{abstract}

We investigate the impact of dissipative dynamic boundary conditions applied at one end of a beam, analyzing their influence on model stability within the Euler-Bernoulli framework. Our primary finding is that hybrid dissipation does not alter the decay characteristics of the original model. We examine two scenarios: first, when hybrid dissipation is the sole dissipative mechanism, and second, when it complements other dissipative mechanisms. In the first case, we demonstrate that hybrid dissipation fails to induce exponential decay, instead producing a slow decay rate of $t^{-1/2}$ for large $t$. In the second case, when acting as a complementary mechanism, hybrid dissipation neither enhances nor diminishes the decay behavior of the original model.

\end{abstract}

	\noindent{\it Keywords and phrases}: Euler Bernoulli equation, Semigroup theory, Exponential stability, Polynomial stability, Lack of the exponential stability.

		\section{Introduction}

We investigate the asymptotic behavior of the hybrid Euler-Bernoulli model with a tip mass. The model consists of the Euler-Bernoulli equation coupled with a dynamic boundary condition. Before addressing our main problem, we provide a brief discussion of the model. The equation of motion is given by
\begin{align}
\rho u_{tt} + M_{xx} &= 0, \quad (x,t) \in (0,\ell) \times \mathbb{R}^{+}_{0}, \label{b1}
\end{align}
where $M = \alpha u_{xx}$ represents the bending moment, $\alpha$ is a positive constant, and $u$ denotes the transverse displacement. The following figure illustrates the problem setup.

%%%%%%%%%%%%%%%%%%%%%%%%%%%%
\begin{figure}[H]
	\setlength{\unitlength}{2pt}
	\begin{center}
		\begin{tikzpicture}[xscale=1,yscale=1]
			\draw[<->,color=gray!50,very thick] (-1,0)--(8.5,0);
			\draw[<->,color=gray!50,very thick] (0.8,-0.5)--(0.8,3.5);
			\draw[fill=blue, opacity=0.5, rounded corners=5mm] (-0.8,1.2) rectangle (0.8,2.8);
			\draw[color=black, very thick, rounded corners=5mm] (-0.8,1.2) rectangle (0.8,2.8);
			\draw[fill=yellow, opacity=0.5] (0.8,1.7) rectangle (7.5,2.3);
			\draw[-] (0.8,2)--(7.5,2);
			\draw[-,dashed] (7.5,2)--(7.5,0) node[below] {$\ell$};
			\draw[-,very thick] (0,2) node {$\bullet$};
			\draw[fill=gray, pattern=north east lines] (7.5,1) rectangle (7.8,3);
			\draw[-] (-1.6,2) node {\scriptsize \it tip body };
			\draw[-] (-1.6,1.5) node {\scriptsize \it of mass $m_0$};
			\draw[-] (0.6,0) node[below] {$0$};
			\draw[-,dashed, thick] (0,2)--(0.8,2);
			\draw[-] (0.4,2) node[below] {$d$};
			\draw[-] (0.1,2.65) node[below] {$O'$};
			\draw[-] (1,2.65) node[below] {$O$};
			\draw[-,very thick] (.8,2) node {$\bullet$};

		\end{tikzpicture}
	\end{center}
	\caption{\it Beam with Tip Body}\label{f1}
\end{figure}
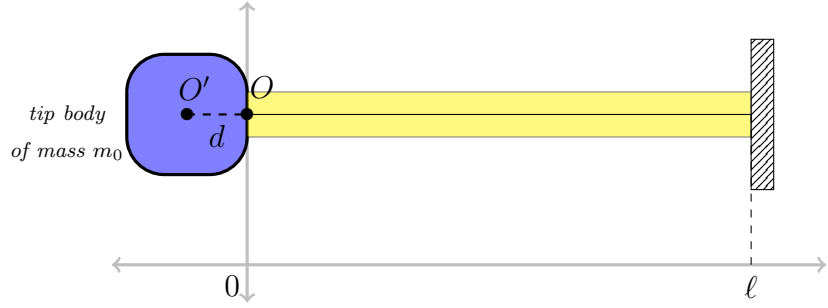

\noindent
We assume that the container is rigidly
attached to the end $x=0$, and that the container and its
contents have mass $m$ and a center of mass $O'$ located a distance 
$d$ from the end of the beam $O$ . We assume that the
damping effect of the internal granular material can be
represented by damping coefficients $\gamma$ and $\gamma^{*}$ whose 
precise contributions are described below.

The force balance at the tip body, located at $x=0$ (see Figure~\ref{f1}), requires special attention. It is given by
\begin{equation}\label{force_balance}
m u_{tt}(0,t) + m d u_{ttx}(0,t) + \gamma u_t(0,t) = -M_x,
\end{equation}
where the vertical component of the inertial term is $m (u(0,t) + d \sin(\beta))$, and $\beta$ represents the angle between the segment $\overline{OO'}$ and the $x$-axis. For small $\beta$, we apply the approximation 
$\sin(\beta)\! \approx \! \tan(\beta)\!  \approx \! u_x(0,t)$ hence $m (u(0,t) + d \sin(\beta))\approx m (u(0,t) + d u_x(0,t) $. The term $\gamma u_t(0,t)$ accounts for damping due to the granular material, with $\gamma$ as the internal damping coefficient. For further details, see~\cite{Sillor1, Sillor2,1739862}.

Next, we consider the moment balance about an axis passing through the center of mass of the tip body, perpendicular to the plane of motion:
\begin{equation}\label{moment_balance}
J u_{xtt}(0,t) = -M + d M_x - d \gamma^* u_{xt}(0,t),
\end{equation}
where $J$ denotes the moment of inertia of the combined mass of the container and granular material about its center of mass, and $\gamma^*$ represents the damping coefficient associated with the angular motion of the granular material.

Combining these, we derive the following system of equations:
\begin{align}
m u_{tt}(0,t) + m d u_{ttx}(0,t) + \gamma u_t(0,t) + \alpha u_{xxx}(0,t) &= 0, \label{Hib1} \\
m d u_{tt}(0,t) + (J + m d^2) u_{ttx}(0,t) + d \gamma u_t(0,t) + d \gamma^* u_{xt}(0,t) - \alpha u_{xx}(0,t) &= 0, \label{Hib2}
\end{align}
where Equation~\eqref{Hib1} follows from Equation~\eqref{force_balance} by substituting $M = \alpha u_{xx}$, and Equation~\eqref{Hib2} is obtained by inserting $M_x$ from Equation~\eqref{force_balance} into Equation~\eqref{moment_balance}.

At the opposite end, $x = \ell$, the boundary conditions are:
\begin{equation}\label{boundary_l}
u(\ell,t) = u_x(\ell,t) = 0.
\end{equation}

Additionally, we prescribe the following initial conditions:
\begin{equation}\label{initial_conditions}
u(x,0) = u_0(x), \quad u_t(x,0) = u_1(x), \quad u_t(0,0) = w_0, \quad u_{xt}(0,0) = w_1,
\end{equation}
for all $x \in (0,\ell)$ and $t > 0$.

The Euler-Bernoulli beam model with an attached hollow-tip body has been studied in \cite{Sillor1, Sillor2}. In \cite{Sillor1}, the authors established the well-posedness and exponential decay of solutions for viscoelastic materials. In \cite{Sillor2}, the authors analyzed the thermomechanical behavior of a viscoelastic beam with an attached tip body, modeling a coupled energy-elastic system for the beam alongside an equation for the temperature of the tip body. They proved the existence and uniqueness of a weak solution and presented numerical simulations of the model.

The absence of exponential stability in the hybrid Euler-Bernoulli model was first addressed in \cite{Litt}, where the authors demonstrated that the system lacks exponential stability. In \cite{Rao}, the author investigated a clamped beam with a mass attached at one end under high-derivative boundary damping. The study showed that uniform stabilization is achieved when standard boundary feedbacks are applied at the end without the mass, but not when applied at the end with the attached mass. This analysis utilized a compact perturbation method to determine conditions for uniform stabilization. For control problems related to the hybrid model, see \cite{Littman-Markus, Raocontrol}. In \cite{1739862}, the authors considered the Euler-Bernoulli model with a tip mass, neglecting the vertical component of the inertial term and incorporating a dissipative boundary condition, under which they proved exponential stability. See also \cite{Conrad}. In \cite{Guo}, B.Z. Guo investigated the boundary control of a hybrid system comprising an Euler-Bernoulli beam with variable coefficients and a linked rigid body under high-derivative boundary damping. The study focused on stability, exact controllability, and exact observability using a Riesz basis approach.

The main contribution of this paper is to demonstrate that the dissipative mechanism introduced by the hybrid model is insufficient to produce exponential stability. We consider the conservative Euler-Bernoulli model with a tip mass, incorporating the effect of rotary inertia. Under these conditions, we prove that the system described by Equations \eqref{b1}, \eqref{Hib1}--\eqref{initial_conditions} is not exponentially stable. Furthermore, we show that the solution decays to zero at a rate of $t^{-1/2}$ as $t \to \infty$ for all initial data in the domain of the operator $\mathcal{A}$, improving the result in \cite{Litt}. Motivated by the applications of viscoelastic and thermoviscoelastic beam models with a tip mass, as studied in \cite{Sillor1, Sillor2}, we compare the asymptotic behavior of the hybrid model (with tip mass) and the corresponding non-hybrid model (without tip mass). Our primary result establishes that the model with a tip mass is exponentially stable if and only if the model without a tip mass is exponentially stable. This indicates that the dissipation induced by the attached mass does not influence the exponential stability of the model.

The remainder of this paper is organized as follows. In Section~\ref{sec:well_posedness}, we establish the well-posedness of the model defined by Equations \eqref{b1}, \eqref{Hib1}--\eqref{initial_conditions} and demonstrate its lack of exponential stability. In Section~\ref{sec:polynomial_stability}, we prove the polynomial stability of the associated semigroup. In Section~\ref{sec:comparison}, we compare the asymptotic behavior of models with and without an attached tip mass, showing the equivalence of exponential stability between their respective semigroups. Finally, in Section~\ref{sec:applications}, we present applications of the results obtained in Section~\ref{sec:comparison}.

\section{The semigroup approach }\label{sec:well_posedness}
\setcounter{equation}{0}
Let us introduce  the energy associated to the above model  
$$
E(t)=\frac 1 2 \int_0^\ell\left( |u_{t}|^2+\alpha |u_{xx}|^2\right)\,dx +\frac{m}{2}|u_{t}(0,t)|^2+md(u_{xt}(0,t)u_{t}(0,t))+	\frac{J+md^2}{2}|u_{xt}(0,t)|^2
$$
It is no difficult to see  that $E$ verifies 
\begin{equation}\label{energy}
			\frac{d}{dt}E(t) =-\gamma|u_t(0,t)|^2-d\gamma^*|u_{xt}(0,t)|^2-d\gamma u_t(0,t)u_{xt}(0,t).
\end{equation}
			Assuming that 
			\begin{equation}\label{hipo1}
d\gamma\leq 2\gamma^*,
\end{equation}
there exists a positive constant $\alpha_0$ such that  
			$$
			\frac{d}{dt}E(t) \leq -\alpha_0(|u_t(0,t)|^2+|u_{xt}(0,t)|^2).
			$$
			Hence integrating over $[0,t]$ we have that 
\begin{equation}\label{dissT}
			E(t) + \int_0^t\alpha_0(|u_t(0,t)|^2+|u_{xt}(0,t)|^2)d\tau \leq E(0).
\end{equation}
To reduce to first order system \eqref{b1}, \eqref{Hib1}-\eqref{Hib2} let us introduce the vector
$$
U=(u,u_t,u_{t}(0,t),u_{xt}(0,t))^\top=(u,v,w,z)^\top
$$
To simplify the model, let us introduce the vector  $V=(u_{t}(0,t),u_{xt}(0,t))^\top=(w(t),z(t))^\top$, hence  the boundary condition \eqref{Hib1}-\eqref{Hib2} can be written as 
\begin{equation}\label{zHib1}
BV_t+KV=\Gamma, 
\end{equation}
where
\begin{equation}\label{MatrC}
B=\begin{pmatrix}
m&md\\
md&J+md^2
\end{pmatrix},\quad 
K=\begin{pmatrix}
\gamma&0\\
d\gamma &d\gamma^* 
\end{pmatrix}, \quad \Gamma=\alpha\begin{pmatrix}
- u_{xxx}(0,t)\\
 u_{xx}(0,t)
\end{pmatrix}.
\end{equation}
Note that the Matrix $B$ and $K$ are positive definite. 	
Let us introduce the space $\mathbf{H}_0$ defined as 
\begin{equation}\label{PhaseT}
\mathbf{H}_0=\mathcal{V}\times L^2(0,\ell)
\end{equation}
	where 
\begin{equation}\label{BcL}
\mathcal{V}=\left\{g\in H^2(0,\ell); \;\; g(\ell)=g_x(\ell)=0\right\}
\end{equation}
Let  $\mathcal{H}$ be the phase space  
\begin{equation}\label{Phase1}
\mathcal{H}=\mathbf{H}_0\times \mathbb{C}^2.
\end{equation}
	that together  with the norm
\begin{equation}\label{norm}
	\|U\|_{\mathcal{H}}^2= \int_0^\ell\left( |v|^2+\alpha |u_{xx}|^2\right)\,dx +BV\cdot V.
\end{equation}
is a Hilbert space. 
Let us define the operator $\mathcal{A}$ as

	\begin{align}\label{Adef1}
				U=\begin{pmatrix}
					u\\
					v\\
					V\
					\end{pmatrix}, \quad \quad \mathcal{A}U=\underbrace{\begin{pmatrix}
					v\\
					-\alpha u_{xxxx} \\
								B^{-1}\Gamma
					\end{pmatrix}}_{\mathcal{A}_0 U}- \underbrace{\begin{pmatrix}
					0\\
					0 \\
					B^{-1}KV
					\end{pmatrix}}_{B_0 U}.
				\end{align}
				With domain 
				$$
				D(\mathcal{A})=\left\{U\in \mathcal{H};\;\; u\in H^4(0,\ell),\; v\in \mathcal{V},\; w=v(0),\quad z=v_x(0)
				\right\}, 
				$$
and $D(\mathcal{A}_0)=D(\mathcal{A})$. Under the above notations, system \eqref{b1}, \eqref{Hib1}-\eqref{Hib2} can be written as 
	\begin{align}\label{problem}
		U_t-\mathcal{A}U=0,\quad U(0)=U_0\in \mathcal{H}.
	\end{align}
	The norm defined in \eqref{norm} induces an inner product  for which the operator \(\mathcal{A}\) and \(\mathcal{A}_0\) are dissipative, that is 
\begin{equation}\label{dissi}
\text{Re} \, (\mathcal{A}U, U)_{\mathcal{H}} \leq -\alpha_0 \left( |w|^2 + |z|^2 \right),\quad \text{Re} \, (\mathcal{A}_0U, U)_{\mathcal{H}} = 0
\end{equation}
	
\noindent 
	The resolvent system is expressed as 
\begin{equation}\label{resolvent}
i\lambda U-\mathcal{A}U=F\quad \in\mathcal{H}
\end{equation}
 which, in terms of its components, is given by
\begin{align}
	i\lambda u-v=f_1, \label{r1}\\
	i\lambda v+\alpha u_{xxxx}=f_2,\label{r2}\\
	i\lambda BV+KV-\Gamma=Bf_4, \label{r6}
\end{align}
where $V=(w,z)^\top$. For $\mathcal{A}_0$ the resolvent system is the same for $\gamma=\gamma^*=0$, which implies that $K=0$. 
Equation \eqref{r6} in terms of the components  is written as 
\begin{eqnarray}
i\lambda mw+ i\lambda md z+\gamma w+\alpha u_{xxx}(0)&=&0\label{zwHib1}\\
i\lambda mdw+i\lambda ({J}+md^2) z+d\gamma w +d\gamma^*z-\alpha u_{xx}(0)&=&0\label{zwHib2}
\end{eqnarray}
Taking inner product  \eqref{resolvent} by $U$ and taking the real  part and using \eqref{dissi} arrive to 
\begin{equation}\label{dissipative}
\gamma|w|^2+\gamma_2|z|^2\leq c\|F\|_\mathcal{H}\|U\|_\mathcal{H}.
\end{equation}

\noindent	
Under the above conditions we have 
	\begin{theorem}\label{semigroup}
	The operator $\mathcal{A}$ and $\mathcal{A}_0$ generate a contractive semigroup. In particular we have that 
	if $U_0\in \mathcal{H}$ then solution of \eqref{problem} satisfies 
	$U\in C([0,T];\mathcal{H})$. Moreover if $U_0\in D(\mathcal{A})=D(\mathcal{A}_0)$ then the corresponding solution  verifies 
$$
	U\in C([0,T];D(\mathcal{A}))\cap C^1([0,T];\mathcal{H})
	$$
\end{theorem}
\begin{proof}
	Since the operator $\mathcal{A}$ is dissipative and $0\in \varrho(\mathcal{A})$, standard procedures (see \cite{Pazy}) shows that $\mathcal{A}$ is the infinitesimal generator of a contractions semigroup. 
\end{proof}

\begin{remark}\label{Rem1}

From Theorem \ref{semigroup}, we conclude that the evolution equation given by \eqref{b1}, \eqref{Hib1}--\eqref{Hib2}admits a mild solution satisfying \eqref{dissT} for any initial data in the phase space \(\mathcal{H}\). The proof relies on density arguments, beginning with initial data in the domain \(D(\mathcal{A})\) and leveraging the density of \(D(\mathcal{A})\) in \(\mathcal{H}\).

Moreover for $\mathcal{A}_0$, that is   $\gamma=\gamma^*=0$  system \eqref{b1}, \eqref{Hib1}--\eqref{Hib2} is conservative and relation \eqref{energy} implies 
$$
E(t)=E(0)
$$
which means that the semigroup $e^{\mathcal{A}_0 t}$ does not decay. That is the type and the essential type of the semigroup $e^{\mathcal{A}_0 t}$ are zero. 
\end{remark}

\bigskip\noindent
Let us denote by 
$$
\mathcal{J}(x,t)= \left| u_{t}(x,t)\right| ^{2}  +\alpha |u_{xx}(x,t)|^2 .
$$
\begin{lemma}\label{lemT} Under hypothesis \eqref{hipo1} we get that 
	the solution of System  \eqref{b1}, \eqref{Hib1}-\eqref{initial_conditions}  verifies 
	\begin{eqnarray}\label{Obs111}
		\left| \ell \int_{0}^{t}\mathcal{J}(0,\tau)d\tau -\int_{0}^{t}\int_{0}^{\ell}\mathcal{J}(x,\tau)+\left| u_{xx}(x,\tau)\right|^{2} dxd\tau \right| &\leq& cTE(0).
	\end{eqnarray}
	
\end{lemma}
\begin{proof}
	Multiplying \eqref{r2} by $q\overline{u_x}$ with $q(x)=x-\ell$ and integrating we have
	\begin{align}
 \int_{0}^{\ell} u_{tt}q  {u_x}\ dx+\int_{0}^{\ell} \alpha u_{xxxx} q{u_x}\ dx&=0.  \label{tt1}
	\end{align}
	Using the resolvent equation \eqref{r1}, we have 
\begin{eqnarray*}
 \int_{0}^{\ell} u_{tt}q   {u_x}\ dx &=&\frac{d}{dt}\int_{0}^{\ell} u_tq  {u_x}\ dx-\int_{0}^{\ell} u_tq  {u_{xt}}\ dx\\
 &=&\frac{d}{dt}\int_{0}^{\ell} u_tq  {u_x}\ dx-\frac 12 \int_{0}^{\ell} q\frac{d}{dx}|u_{t}|^2\ dx
\end{eqnarray*}
Since $u_x(\ell)=0$ we get 
\begin{eqnarray*}
	\int_{0}^{\ell} u_{xxxx} q\overline{u_x}\ dx&=&-\dfrac{\ell}{2} u_x(0)u_{xxx}(0)+u_{xx}(0)u_x(0)+\int_{0}^{\ell}|u_{xx}|^2dx-\int_{0}^{\ell}u_{xxx}q \overline{u_{xx}}\!dx\\
&=&-\dfrac{\ell}{2} u_x(0)u_{xxx}(0)+u_{xx}(0)u_x(0)+\int_{0}^{\ell}|u_{xx}|^2dx-\frac 12 \int_{0}^{\ell}q \frac{d}{dx}|u_{xx}|^2 dx
\end{eqnarray*}
Substitution of the above equations into \eqref{tt1} and recalling the definition of $\mathcal{J}$ we get

	\begin{align}
		-\dfrac{1}{2}\int_{0}^{\ell} q\dfrac{d}{dx}\mathcal{J}(x)  dx +\alpha \int_{0}^{\ell}|u_{xx}|^2dx &=\frac{d}{dt}\int_{0}^{\ell} u_tq  {u_x}\ dx
		\nonumber\\
		&-\dfrac{\ell}{2}\alpha u_x(0)u_{xxx}(0)+\alpha u_{xx}(0)u_x(0)\label{Fineq}
	\end{align}
Using   \eqref{Hib1}--\eqref{Hib2}  
and recalling  $(u_{t}(0,t),u_{xt}(0,t))=(w(t),z(t))$ 
\begin{align}
 -u_x(0)u_{xxx}(0)&= u_{x}(0,t)\left(mw_{t}+md z_{t}+\gamma w\right)\nonumber\\
&= \frac{d}{dt}[u_{x}(0,t)\left(mw+md z\right)]-u_{xt}(0,t)\left(mw+md z\right)+\gamma u_{x}(0,t)w,\nonumber\\
 &=\frac{d}{dt}\underbrace{[u_{x}(0,t)\left(mw+md z\right)]}_{:=X_1}-mzw- m d \left| z\right|^2 +\gamma u_{x}(0,t)w.\label{Obb1}
	\end{align}
Using the identity for $u_{xx}(0,t)$ given by \eqref{Hib2}
\begin{eqnarray}
u_x(0)u_{xx}(0) &=&u_x(0,t)\left(mdw_{t} +({J}+md^2) z_{t}+d\gamma w+d\gamma^*z\right)\nonumber\\
&=&	\frac{d}{dt}[u_{x}(0,t)\left(mdw+({J}+md^2) z\right)]-mdu_{xt}(0,t)w-({J}+md^2) |z|^2\nonumber\\
&&+d\gamma u_x(0,t)w+d\gamma^*u_x(0,t)u_{xt}(0,t) \nonumber\\
&=&	\frac{d}{dt}\underbrace{[u_{x}(0,t)\left(mdw+({J}+md^2) z\right)]}_{:=X_2}-mdwz-({J}+md^2) |z|^2\nonumber\\
&&+d\gamma u_x(0,t)w+\frac{d\gamma^*}{2}\frac{d}{dt}\underbrace{|u_x(0,t)|^2}_{:=X_3} ,\label{Obb2}
	\end{eqnarray}	
Denoting by $X(t)=\alpha X_1+\alpha X_2+\alpha X_3$, that is  
$$
X(t)=\frac{\ell \alpha}{2} u_{x}(0,t)\left(mw+md z\right) +\alpha u_{x}(0,t)\left(mdw+({J}+md^2) z\right)+\alpha\frac{ d\gamma^*}{2} |u_x(0,t)|^2 .
$$
Summing up identities \eqref{Obb1} and \eqref{Obb2} yields 
\begin{equation}\label{above} 
-\dfrac{\ell}{2}\alpha u_x(0)u_{xxx}(0)+\alpha u_{xx}(0)u_x(0)=\frac{d}{dt}X(t)+\mathfrak{R}, 
\end{equation}
	where 
	$$
	 \mathfrak{R}=-m|z|^2- 2md wz-({J}+md^2) |z|^2+d\gamma u_x(0,t)w-\alpha \gamma z u.
	 $$
Substitution of identity \eqref{above} into \eqref{Fineq} we get 
	\begin{align}
 -\int_{0}^{\ell} q\dfrac{d}{dx}\mathcal{J}(x)dx+\int_{0}^{\ell}\left| u_{xx}\right| ^{2} dx&= \frac{d}{dt}\int_{0}^{\ell} u_tq  {u_x}\ dx+\frac{d}{dt}X(t)+\mathfrak{R}.\label{casi}
	\end{align}
	 	Performing and integration by parts in $[0,\ell]$ and  integrating with respect to $t$ over $[0,t]$ 
we get 

		\begin{eqnarray}
&& \ell \int_{0}^{t}\mathcal{J}(0,\tau)d\tau -\int_{0}^{t}\!\!\int_{0}^{\ell}\mathcal{J}(x,\tau)+\left| u_{xx}(x,\tau)\right|^{2} dxd\tau =\mathfrak{X}(t)-\mathfrak{X}(0)+\int_0^t \mathfrak{R}ds\label{TTcasi}
	\end{eqnarray}
where 
$$
\mathfrak{X}(t)= \int_{0}^{\ell} u_tq  {u_x}\ dx+X(t)
$$	
Using Sobolev's and  Poincare's inequality and \eqref{dissT} we get 	 
\begin{align}\label{eqq1}
	|u_x(0,t)|^2\leq c\int_{0}^{\ell}| u_{xx}| ^{2} dx,\quad |u(0,t)|^2\leq c\int_{0}^{\ell}| u_{x}| ^{2} dx 
\end{align}
	So we arrive to 
\begin{equation}\label{eqq2}
	d\gamma u_x(0,t)w\leq \frac{\epsilon}{c}	|u_x(0,t)|^2+c_\epsilon|w|^2
\end{equation}
Hence we have 
$$
\mathfrak{X}(t)\leq c_0E(0),\quad \left|\int_0^t \mathfrak{R}\;ds\right|\leq c_1\int_0^t  |z|^2+ |w|^2\;dt+\epsilon\int_{0}^{\ell}| u_{xx}| ^{2} dx
$$	
Inserting the above inequalities into \eqref{TTcasi} yields
	\begin{eqnarray*}
&&\left| \ell \int_{0}^{t}\mathcal{J}(0,\tau)d\tau -\int_{0}^{t}\!\!\int_{0}^{\ell}\mathcal{J}(x,\tau)+\left| u_{xx}(x,\tau)\right|^{2} dxd\tau \right|\\
&&\qquad\leq c_0E(0)+c_1\int_0^t  |z|^2+ |w|^2\;dt+\epsilon\int_{0}^{\ell}| u_{xx}| ^{2} dx.
	\end{eqnarray*}	
Using \eqref{dissT} our conclusion follows.

\end{proof}

\bigskip
\noindent
The next Lemma will play an important role in that follows

 \newcommand{\zz}{\quad \rightarrow\quad}
\begin{theorem}\label{Kim}
Let  $\Omega $ be an open bounded set of $\mathbb{R}^n$ and $\mathcal{O}$ be the set
$$
\mathcal{O}=\left\{w\in C(0,T;H^\beta(\Omega));\; w_t\in C(0,T;H^\alpha(\Omega))
\right\}.
$$
Let us suppose that $0\leq \alpha \leq r<\beta$ then any closed bounded set $B\subset \mathcal{O}$ is compact in $C(0,T;H^r(\Omega))$.  
\end{theorem}
\begin{proof} Here we use Arselá Ascoli Theorem. 
Since $B$ is bounded by hypotheses, we only need to prove the equicontinuity of the elements in $B$. Indeed Let us take $u\in B$ and 
using interpolation for  $r=\theta\alpha+(1-\theta)\beta$ we have, 
\begin{eqnarray*}
\|u(t_2)-u(t_1)\|_{H^r(\Omega)}&\leq&
C_\theta \|u(t_2)-u(t_1)\|_{H^\alpha(\Omega)}^\theta
\|u(t_2)-u(t_1)\|_{H^\beta(\Omega)}^{1-\theta}
\end{eqnarray*}
Since $B$ is bounded exists a positive constant $c$ such that 
$$
\|u(t_2)-u(t_1)\|_{H^\beta(\Omega)}\leq c_\beta
$$
using 
$$
u(t_2)-u(t_1)=\int_{t_1}^{t_2}u_{t}(s)\;ds
$$
So we have that
\begin{eqnarray*}
\|u(t_2)-u(t_1)\|_{H^r(\Omega)}&\leq&
C_\theta \|u_{\nu}(t_2)-u_{\nu}(t_1)\|_{H^\alpha(\Omega)}^\theta
\\
&\leq&
C_\theta \{\int_{t_1}^{t_2}
\|u_{t}(t)\|_{H^\alpha(\Omega)}\;dt\}^\theta\\
&\leq&C_\theta\|u_t\|_{C(0,T;H^\alpha(\Omega))}^{\theta}|t_1-t_2|^{\frac{\theta}{2}}
\end{eqnarray*}
Finally, note that for any $t\in [0,T]$ and any $u\in B$ we have that $u(t)$ lies in a compact set of $H^r(\Omega)$ for $r<\beta$.  Hence using the Arsela-Ascoli Theorem our conclusion follows. 
\end{proof}

\bigskip
\noindent
Under the above conditions we have 
\begin{theorem}\label{nonExp}
System   \eqref{b1}, \eqref{Hib1}-\eqref{initial_conditions}  is not exponentially stable. 
\end{theorem}
\begin{proof}
	To demonstrate the lack of exponential stability of the semigroup \( e^{\mathcal{A}t} \), we show that the difference \( e^{\mathcal{A}t} - e^{\mathcal{A}_0 t} \) is a compact operator, which, by Remark \ref{Rem1}, implies that the growth bound of \( e^{\mathcal{A}t} \) is zero. To this end, we utilize the decomposition provided in \eqref{Adef1}, which allows us to express the solution of the evolution problem \eqref{problem} as \( U_t - \mathcal{A}_0 U = -B_0 U \), that is

		\begin{align}\label{nose}
		U=e^{\mathcal{A}_0 t}U_0-\int_{0}^{t}e^{\mathcal{A}_0(t-s)} B_0U \ ds
	\end{align}
Since $U$ is the solution of \eqref{problem} we have $	U=e^{\mathcal{A} t}U_0$, hence identity \eqref{nose} can be written as 
		\begin{align*}
e^{\mathcal{A} t}U_0-e^{\mathcal{A}_0 t}U_0=-\int_{0}^{t}e^{\mathcal{A}_0(t-s)} B_0U \ ds.
	\end{align*}
	To show the lack of exponential estability we only need to proof that the operator $B_0$ is compact over $\mathcal{H} $. Let us denote by $\Phi(t)=mdu_{t}(0,t)+(J+md^2) u_{tx}(0,t)$. Using equation \eqref{Hib2} we have that 
	$$
	\Phi_t=-d\gamma u_t(0,t)-d\gamma^*u_{xt}(0,t)+\alpha u_{xx}(0,t)
	$$
	Using  \eqref{dissT}  and Lemma \ref{lemT} we conclude that $\Phi$ belong to a bounded set in $ H^1(0,T)$ for any initial data $U_0\in B_R(0)\subset \mathcal{H}$, where $B_R(0)=\{w\in\mathcal{H};\;\|w\|_{\mathcal{H}}\leq R\}$. 
Hence 
\begin{equation}\label{compc1}
mdu_{t}(0,t)+(J+md^2) u_{tx}(0,t)=F_1 \quad\mbox{is bounded }\quad H^1(0,T),
\end{equation}
for any initial data $U_0\in B_R(0)\subset  \mathcal{H}$. On the other hand, integrating \eqref{problem} with respect to $t$ we get 
$$
U(t)-U(0)-\mathcal{A}\int_0^tU(s)\;ds=0. 
$$
Since $0\in\varrho(\mathcal{A})$  we conclude that for $U_0 \in \mathcal{H}$ there exists $\Theta\in D(\mathcal{A})$ such that $\mathcal{A}\Theta=U_0$. Denoting by  $\mathfrak{U}(t)=\int_0^tU(s)\;ds+\Theta$, 
  we have that $\mathfrak{U}$ satisfies 
\begin{align*}
		\mathfrak{U}_t-\mathcal{A}\mathfrak{U}=0,\quad \mathfrak{U}(0)=\Theta\in D(\mathcal{A}).
	\end{align*}
Therefore from Theorem \ref{semigroup} we conclude that 
	$$
\mathfrak{U}\in C([0,T];D(\mathcal{A}))\cap C^1([0,T];\mathcal{H}).
	$$
	In particular $\widehat{u}=\int_0^tu\; dt\in C(0,T;H^4(0,\ell))$ and $\widehat{u}_t= u \in C(0,T;H^2(0,\ell))$. Note that  
	
	$$
	H^4(0,\ell)\subset H^{4-\delta}(0,\ell)\subset H^{2}(0,\ell),
	$$
	 for $\frac 12\leq \delta<1$, Using Lemma \ref{Kim} 
	$$\mathcal{O}=\left\{f\in C(0,T;H^4(0,\ell)),\;\;\frac{df}{dt}\in C(0,T;H^2(0,\ell))\right\}
	$$
has compact embeeding in $C(0,T;H^{4-\delta}(0,\ell))\subset C(0,T;C^{3}(0,\ell))$. Since $\widehat{u}\in \mathcal{O}$ for any initial data $U_0\in B_R(0)\subset  \mathcal{H}$, we have that $\widehat{u}_{xxx}$
belong to a compact set of $C(0,T;H^{\delta}(0,\ell))$, for  $\frac 12 \leq \delta <1$. In particular 
$\widehat{u}_{xxx}$ belongs to a compact set in $C(0,T;C(0,\ell))$.
Integrating \eqref{Hib1} we get 
\begin{equation}\label{compc}
mu_{t}(0,t)+md u_{tx}(0,t)=-\gamma u (0,t)-\alpha \int_0^tu_{xxx}(0,\tau)d\tau. 
\end{equation}
The right hand side of \eqref{compc} is compact in $L^2(0,T)$. Therefore 
\begin{equation}\label{compcy}
mu_{t}(0,t)+md u_{tx}(0,t) =F_2 \quad\mbox{is compact in }\quad L^2(0,T).
\end{equation}
Solving \eqref{compc1} and \eqref{compcy}  we conclude that $u_{t}(0,t)$ and $u_{tx}(0,t)$ lies in a compact set $L^2(0,T)$ for any initial data $U_0\in B_R(0)\subset  \mathcal{H}$. 
Hence the operator $B_0$ is compact. The proof is now complete. 
\end {proof}
	
\section{Polynomial decay}	\label{sec:polynomial_stability}
\setcounter{equation}{0}
Here we show that the solution of system    \eqref{b1}, \eqref{Hib1}-\eqref{initial_conditions}  decays polynomially to zero. 
Let us denote by 
$$
\mathcal{J}(x)= \left| v(x)\right| ^{2}  +\alpha |u_{xx}(x)|^2 .
$$
Under the above notations we have 
\begin{lemma}\label{lem1}
	The solution of the resolvent system \eqref{r1}-\eqref{r6} verifies 
	\begin{eqnarray}\label{Obs131}
		\left|\frac\ell 2\mathcal{J}(0) -\int_{0}^{\ell}\mathcal{J}(x)+\left| u_{xx}\right|^{2} dx\right| &\leq& c\left\| U\right\|_\mathcal{H} \left\| F\right\|_\mathcal{H} .
	\end{eqnarray}
	
\end{lemma}
\begin{proof}
	Multiplying \eqref{r2} by $q\overline{u_x}$ with $q(x)=x-\ell$ and integrating we have
	\begin{align} \label{Eqq11}
		i\lambda\int_{0}^{\ell} vq   \overline{u_x}\ dx+\int_{0}^{\ell}\alpha u_{xxxx} q\overline{u_x}\ dx&=\int_{0}^{\ell}f_2q\overline{u_x}\ dx.  
	\end{align}
Using the resolvent equation \eqref{r1}, we have 
\begin{eqnarray}
		i\lambda\int_{0}^{\ell} vq   \overline{u_x}\ dx&=&		-\dfrac{1}{2}\int_{0}^{\ell} \!\!\!q\dfrac{d}{dx}\left|v \right|^{2}  dx-\int_{0}^{\ell} \! \!qv\overline{f}_{1,x} \; dx \label{inq1}
\end{eqnarray}
Moreover since $u_x(\ell)=0$ we get 
\begin{eqnarray}
	\int_{0}^{\ell} u_{xxxx} q\overline{u_x}\ dx&=&\dfrac{\ell}{2} u_x(0)u_{xxx}(0)-u_{xx}(0)u_x(0)+\int_{0}^{\ell}|u_{xx}|^2dx-\int_{0}^{\ell}u_{xxx}q \overline{u_{xx}}\!dx\nonumber\\
&=&\dfrac{\ell}{2} u_x(0)u_{xxx}(0)-u_{xx}(0)u_x(0)+\int_{0}^{\ell}|u_{xx}|^2dx\nonumber\\
&&-\frac 12 \!\!\int_{0}^{\ell}\!q \frac{d}{dx}|u_{xx}|^2 dx\label{inq2}
\end{eqnarray}
Substitution of the above equations into \eqref{Eqq11} and recalling the definition of $\mathcal{J}$ we get

	\begin{align}
		-\dfrac{1}{2}\int_{0}^{\ell} q\dfrac{d}{dx}\mathcal{J}(x)  dx +\int_{0}^{\ell}\alpha |u_{xx}|^2dx &=\underbrace{\int_{0}^{\ell} \! \!q\left(v\overline{f}_{1,x} \!+\!f_2\overline{u_x}\right) dx}_{:=I_3}
		\nonumber\\
		&+\dfrac{\ell}{2}u_x(0)u_{xxx}(0)-u_{xx}(0)u_x(0)\label{Fineq1}
	\end{align}
Using \eqref{r1},  \eqref{dissipative} and recalling that  $V=(v(0),v_{x}(0))^\top=(z,w)^\top$ we get 
$$
i\lambda u(0)-v(0)=f_1(0),\quad i\lambda u_x(0)-v_x(0)=f_{1,x}(0),\quad 
$$
or 
$$
i\lambda u(0)-w=f_1(0),\quad i\lambda u_x(0)-z=f_{1,x}(0).
$$
To estimate the boundary terms we use the above equations together with inequality  \eqref{dissipative}
and equation \eqref{zwHib2}
\begin{eqnarray*}
\left| u_x(0)u_{xxx}(0)\right|&=&\left| \frac{1}{i\lambda}(z+f_{1,x}(0))\right|\left|u_{xxx}(0)\right|,\\
&\leq &\left( \frac{c}{|\lambda|}\|F\|_\mathcal{H}\|U\|_\mathcal{H}+\frac{c}{|\lambda|}\|F\|_\mathcal{H}^2\right)^{1/2}
	\left({|\lambda|}\|F\|_\mathcal{H}\|U\|_\mathcal{H}+\|F\|_\mathcal{H}^2\right)^{1/2},\\
	&\leq &\left(\|F\|_\mathcal{H}\|U\|_\mathcal{H}+\|F\|_\mathcal{H}^2\right)^{1/2}
	\left(\|F\|_\mathcal{H}\|U\|_\mathcal{H}+\frac c{\lambda}\|F\|_\mathcal{H}^2\right)^{1/2},\\
	&\leq &c \|U\|_\mathcal{H}\|F\|_\mathcal{H}+c\|F\|_\mathcal{H}^2 ,
	\end{eqnarray*}
Using the same arguments we conclude that  
\begin{eqnarray*}
\left|u_x(0)u_{xx}(0)\right|&\leq &c \|U\|_\mathcal{H}\|F\|_\mathcal{H}+c\|F\|_\mathcal{H}^2 ,
	\end{eqnarray*}	
Moreover we have that 
$$ |I_3|\!\leq\!
	c\left\|U \right\|_\mathcal{H}\left\|F \right\|_\mathcal{H}.
	$$
Substitution of the above inequalities into \eqref{Fineq} we get 
	\begin{align*}
		\left|-\int_{0}^{\ell} q\dfrac{d}{dx}\mathcal{J}(x)dx+\int_{0}^{\ell}\left| u_{xx}\right| ^{2} dx\right|\leq c\left\|U \right\|_\mathcal{H}\left\|F \right\|_\mathcal{H}.
	\end{align*}
Integrating by parts we arrive to  \eqref{Obs131}. 
\end{proof}

\begin{theorem}
The semigroup associated with to system   \eqref{b1}, \eqref{Hib1}-\eqref{initial_conditions}   is polynomially stable with a rate of $t^{-1/2}$.

\end{theorem}

\begin{proof} From \eqref{Obs131} we have that 
	\begin{eqnarray}\label{Ineqqww}
\int_{0}^{\ell}\mathcal{J}(x)+\left| u_{xx}\right|^{2} dx &\leq& \frac\ell 2\mathcal{J}(0)+ c\left\| U\right\|_\mathcal{H} \left\| F\right\|_\mathcal{H} .
	\end{eqnarray}

	Using the resolvent equation \eqref{zwHib2} together with \eqref{dissipative} we get 
	\begin{eqnarray*}\label{JO}
\mathcal{J}(0)&=& \left| v(0)\right| ^{2}  +\alpha |u_{xx}(0)|^2 \\
&=&\left| z\right| ^{2}+\frac{1}\alpha | i\lambda mdz+i\lambda ({J}+md^2) w+d\gamma z+ \gamma^*w|^2\\
&\leq&c\| U\|_{\mathcal{H}}\| F\|_{\mathcal{H}}+c|\lambda|^2\| U\|_{\mathcal{H}} \| F\|_{\mathcal{H}}.
	\end{eqnarray*}
So from \eqref{Ineqqww} we conclude that 
$$
\| U\|_{\mathcal{H}}^2=\int_{0}^{\ell}\mathcal{J}(x)+\left| u_{xx}\right|^{2} dx+BV\cdot V\leq c|\lambda|^2\| U\|_{\mathcal{H}} \| F\|_{\mathcal{H}}
$$
For $\lambda$ large enough. The above inequality implies that 
$$
\| U\|_{\mathcal{H}}^2\leq c|\lambda|^4 \| F\|_{\mathcal{H}}^2
$$
	Hence, our conclusion follows.
\end{proof}

		\section{Comparison between  hybrid and non hybrid model}\label{sec:comparison} 
\setcounter{equation}{0}

	Next we investigate whether dissipative systems lacking exponential stability can achieve exponential decay by incorporating an additional dissipative mechanism into the dynamic boundary conditions. Specifically, we explore whether weakly dissipative systems, which do not exhibit exponential stability, can be transformed into strongly dissipative systems capable of exponential decay through the introduction of hybrid dissipation at the boundary. The answer to this is negative. We will deal with this problem in the context of Euler Bernoulli model. The following theorem will be key in that follows. 

\begin{theorem}\label{TNovo}  Let $S(t)=e^{{\mathcal{A}}t}$ be a
			$C_0$-semigroup of contractions on Banach space. Then, $S(t)$ is
			exponentially stable if and only if
			\begin{equation}\label{hyp}i\mathbb{R}\subset \varrho(\mathcal{A})
			\quad\text{and}\quad
			\omega_{ess}(S(t))<0,
			\end{equation}
where $\omega_{ess}(S(t))$ is the essential growth bound of the semigroup $S(t)$.
		\end{theorem}
		\begin{proof}Here we use \cite[Corollary~2.11, p.~258]{EngelNagel} establishing that the type $\omega$ of the semigroup $e^{\mathcal{A}t}$ verifies
			\begin{equation}\label{idxx}
			\omega=\max\{\omega_{ess}, \omega_\sigma(\mathcal{A})\},
			\end{equation}
			where %$\omega_{ess}$ is the essential type of the semigroup
%and
$\omega_\sigma(\mathcal{A})$ is the upper bound of the spectrum of $\mathcal{A}$. 	Moreover, for any $c>\omega_{ess}$, the set $\mathcal{I}_c:=\sigma(\mathcal{A})\cap\{\lambda\in \mathbb{C}:\;\; \mbox{Re}\lambda \geq c\}$ is  finite.

Let us suppose that \eqref{hyp} is valid.
Since the essential type of the semigroup $\omega_{ess}$ is negative, identity \eqref{idxx} states that the type of the semigroup will be negative provided $\omega_\sigma(\mathcal{A})<0$. 

If
$ \omega_\sigma (\mathcal {A}) \leq \omega_ {ess} $ then we have nothing to prove. Let us suppose that $ \omega_ \sigma (\mathcal {A})> \omega_ {ess} $.  From \eqref{hyp} and  Hille-Yosida Theorem  we have $ \overline{\mathbb {C} _ +} \subset \varrho (\mathcal{A}) $, hence
$ \omega_\sigma (\mathcal {A})\leq 0 $.  On the other hand 
$ \mathcal{I}_ {\omega_{ess}+\delta} $ is finite for $\delta>0$ verifying  $ \omega_ {ess} + \delta <0 $ and $ \omega_ {ess} + \delta <\omega_ \sigma (\mathcal {A}) $. Therefore  we have 
$$
\omega_ \sigma (\mathcal {A})=\sup \mbox{Re }\sigma(\mathcal{A})=\sup \mbox{Re } \mathcal{I}_ {\omega_{ess}+\delta}<0.
$$ 
Hence, the sufficient condition follows.

Reciprocally, let us suppose that the semigroup $S(t)$ is exponentially stable, in particular it goes to zero. Then, by   \cite[Theorem~1.1]{DuyBatt}  we have that $i\mathbb{R}\subset \varrho(\mathcal{A})$. Moreover, since the type $\omega$ verifies \eqref{idxx}, we have that
			$$
			\omega_{ess}\leq \max\{\omega_{ess}, \omega_\sigma(\mathcal{A})\}=\omega<0.
			$$
			Then, our conclusion follows.
\end{proof}
Note that the above characterization is valid for any Banach space.

	%%%%%%%%%%%%%%%%%%%%%%%%%%%

\subsection*{Family of Euler-Bernoulli beam models }	
Here we analyze a family of Euler-Bernoulli beam models as visco elastic and thermo elastic and compare them with their corresponding extensions within a hybrid model framework. Specifically, we consider the governing equation \eqref{TTrge.1} with various constitutive laws:

\begin{equation}\label{TTrge.1}
\rho u_{tt} + M_{xx} = 0, \quad (x,t) \in ]0,\ell[ \times \mathbb{R}^{+},
\end{equation}
subject to the boundary conditions:

\begin{equation}\label{bcnh}
M(0,t) = M_x(0,t) = 0, \quad u(\ell,t) = u_x(\ell,t) = 0, \quad t \geq 0.
\end{equation}
As examples, we consider the following constitutive laws for the moment \(M\):
\begin{itemize}
    \item The conservative elastic model, where \(M\) is defined as in Section 2.
    \item The viscoelastic model of Kelvin-Voigt type, where \(M = \alpha u_{xx} + \alpha_0 u_{xxt}\).
    \item The thermoelastic model, where \(M = \alpha u_{xx} + m \theta\).
\end{itemize}

In the thermoelastic case, the temperature \(\theta\) is incorporated into the model, expanding the phase space to \(\mathbf{H} = \mathbf{H}_0 \times L^2(0,\ell)\), where $\mathbf{H}_0 $ is defined in \eqref{PhaseT}. The corresponding phase space vector is \(\mathcal{U} = (u, v, \theta) \in \mathbf{H}\), where \(v = u_t\) is introduced to reformulate the Euler-Bernoulli equation as a first-order system.

In this section, we define a family of Euler-Bernoulli models as follows. Let \(\mathbf{H} = \mathbf{H}_0 \times \mathbf{H}_1\) denote the phase space with norm $\|\cdot\|_{\mathbf{H}}$, where \(\mathbf{H}_0\) is defined in \eqref{PhaseT}, and \(\mathbf{H}_1\) is a suitable Hilbert space which is the complement when we consider thermo elastic models for example. We denote by \(\mathcal{A}_T\) the infinitesimal generator of a contraction semigroup \(\mathbf{T}(t) = e^{t \mathcal{A}_T}\), defined on \(\mathbf{H}\) with domain \(D(\mathcal{A}_T) \subset \mathbf{H}\). 
So we have that 
\begin{equation}\label{visco2}
 \operatorname{Re}\langle \mathcal{A}_{T} \mathcal{U},\mathcal{U}\rangle_{\mathbf{H}} 
  \leq 0, \quad \forall \; \mathcal{U}\in  D(\mathcal{A}_T). 
\end{equation}
Let us introduce the projection operator \(\mathbb{P}_0: \mathbf{H} \to \mathbf{H}_0\) defined as:

\begin{equation}
\mathbb{P}_0\; \mathcal{U} = \begin{pmatrix}
u \\
v
\end{pmatrix}, \quad \forall \mathcal{U} \in \mathbf{H}.
\end{equation}
Here we assume that the domain of the operator \(\mathcal{A}_T\) satisfies:

\begin{equation}
D(\mathcal{A}_T) \subset \left\{ \mathcal{U} \in \mathbf{H} : M \in H^2(0,\ell), \, u, v \in \mathcal{V}, \, \text{satisfying \eqref{bcnh}} \right\}.
\end{equation}
The corresponding abstract equation 
\begin{equation}\label{GODE-T}
 \frac{d}{dt}\mathcal{U}(t)=\mathcal{A}_{T}\,\mathcal{U}(t),\quad \mathcal{U}(0)=\mathcal{U}_0\in \mathbf{H}
\end{equation}

\subsection*{Extension to hybrid models}
We extend model system \eqref{TTrge.1}--\eqref{bcnh} 
to the hybrid model by changing the boundary condition at the end $x=0$ in  \eqref{bcnh} to a dynamical boundary condition 
where $u(0,t)$ and $u_x(0,t)$ are given by the solution of  the system 
\begin{eqnarray}
mu_{tt}(0,t)+md u_{ttx}(0,t)+\gamma u_t(0,t)+M_{x}(0,t)&=&0\label{GHib1}\\
mdu_{tt}(0,t)+(J+md^2) u_{ttx}(0,t)+d\gamma u_t(0,t)+d\gamma^*u_{xt}(0,t)-M(0,t)&=&0\label{GHib2}
\end{eqnarray}
and at the end $x=\ell$ we leave the same boundary condition as in \eqref{bcnh} 
\begin{equation}\label{bcnhh}
 u(\ell,t)=u_x(\ell,t)=0,\quad t\geq 0
\end{equation}
As in section 3,  we  introduce the vector  $V=(u_{t}(0,t),u_{xt}(0,t))^\top=(w(t),z(t))^\top$, hence  the boundary condition \eqref{GHib1}-\eqref{GHib2} can be written as 
\begin{equation}\label{yzHib1}
BV_t+KV=\Gamma(\mathcal{U}), 
\end{equation}
where as in \eqref{MatrC} we have 
\begin{equation}\label{yMatrC}
B=\begin{pmatrix}
m&md\\
md&J+md^2
\end{pmatrix},\quad 
K=\begin{pmatrix}
\gamma&0\\
d\gamma &d\gamma^* 
\end{pmatrix}, \quad \Gamma(\mathcal{U})=\alpha\begin{pmatrix}
- M_{x}(0,t)\\
M(0,t)
\end{pmatrix}.
\end{equation}
The operator $\Gamma$ is a trace operator associated to the flector moment and the shear force. 
Let us introduce the phase space $\mathcal{H}$ given by 
\begin{equation}\label{PhaseH2}
\mathcal{H}= \mathbf{H}\times \mathbb{C}^2.
\end{equation}
Denoting  $U:=(\mathcal{U}, {V})^\top$, where $\mathcal{U}\in \mathbf{H}$ and ${V}:=(z,w)\in \mathbb{C}^2$, we define the  norm of $\mathcal{H}$ 
$$
\|U\|_{\mathcal{H}}^2=\|\mathcal{U}\|_{\mathbf{H}}^2+B V\cdot V
$$
let us introduce the operator $\mathcal{A}$ associated to the hybrid model 

\begin{equation}\label{DefAT}
\mathcal{A}
:=
\begin{pmatrix}
\mathcal{A}_{T}&\mathbf{0}\\
B^{-1}\Gamma&B^{-1}K
\end{pmatrix}
\end{equation}
hence for $U:=(\mathcal{U}, {V})^\top$  we get
$$
\mathcal{A}U=\begin{pmatrix}
\mathcal{A}_{T}\mathcal{U}\\
B^{-1}\Gamma \mathcal{U}+B^{-1}KV
\end{pmatrix}=\begin{pmatrix}
\mathcal{A}_{T}\mathcal{U}\\
0
\end{pmatrix}+\begin{pmatrix}
0\\
B^{-1}KV
\end{pmatrix}+\begin{pmatrix}
\mathbf{0}\\
B^{-1}\Gamma \mathcal{U}
\end{pmatrix},\quad \forall U\in D(\mathcal{A}).
$$
Here we assume 
\begin{description}
\item{\bf H1 :}
$\mathcal{A}$ is  the infinitesimal generator of a contractions semigroups $\mathcal{T}(t)=e^{t\mathcal{A}}$.
\item{\bf H2 :} The operator  $\mathcal{A}$ verifies the dissipative condition
\begin{equation}\label{visco2}
 \operatorname{Re}\langle \mathcal{A} U,U\rangle_{\mathcal{H}} 
  \leq -\gamma |\widetilde{w}|^2 - \gamma_2 |\widetilde{z}|^2 \leq 0, \quad \forall \; U\in  D(\mathcal{A}). 
\end{equation}
\item{\bf H3 :} The domain of the operator \(\mathcal{A}\) satisfies:
\begin{equation}\label{DomA}
D(\mathcal{A}) \subset \left\{ U \in \mathcal{H} : M \in H^2(0,\ell), \, u, v \in \mathcal{V}, \, \text{satisfying \eqref{GHib1},\eqref{GHib2}, \eqref{bcnhh}} \right\}.
\end{equation}

\end{description}
The Cauchy problem associated to the hybrid model is given by 
$$
U_t-\mathcal{A}U=0, \quad U(0)=U_0.
$$

The main objective of this section is to show that the semigroup $\mathbf{T}(t) = e^{t \mathcal{A}_T}$ is exponentially stable
if and only the semigroup $\mathcal{T}(t)=e^{t\mathcal{A}}$ is also exponentially stable. This means that the dissipation produced by the ODE in \eqref{GHib1}-\eqref{GHib2}, of the hybrid model, is not relevant to stabilize exponentially the system.
Let us introduce the space
\begin{equation}\label{Phase2}
\widetilde{\mathbf{H}}=\mathbf{H}\times \{0\}\times \{0\},
\end{equation}
intended as the extended phase space.
Let us denote by $\Pi$ the projection of $\mathcal{H}$ onto $\widetilde{\mathbf{H}}$:
\begin{equation}\label{Phase3}
\Pi(\mathcal{U},{V}) =(\mathcal{U},0).
\end{equation}

Under the above conditions we can state the following Lemma:
 \begin{lemma}\label{T000}
 Let us suppose that  hypotheses {\bf H1}, {\bf H2}, {\bf H3} and   $0\in \varrho(\mathcal{A})$ then 
the difference $\mathcal{T}(t)-\mathbf{T}(t)\Pi$ is a compact operator over $\mathcal{H}$. Hence the corresponding essential types $\omega_{\text{ess}}(\mathcal{T})$ and $\omega_{\text{ess}}(\mathbf{T}(t)\Pi)$  are equal.
\end{lemma}
\begin{proof}
Note that the solution of $U_t-\mathcal{A}U=0$, $U(0)=U_0$ can be written as
$$
\begin{pmatrix}
\mathcal{U}\\
V
\end{pmatrix}_t=\begin{pmatrix}
\mathcal{A}_{T}\mathcal{U}\\
0
\end{pmatrix}+\begin{pmatrix}
0\\
B^{-1}KV
\end{pmatrix}+\begin{pmatrix}
\mathbf{0}\\
B^{-1}\Gamma \mathcal{U}
\end{pmatrix}
$$
with $ U_0=(\mathcal{U}_0, V_0)^\top$ which implies that
$$
\mathcal{U} =e^{t\mathcal{A}_{T}}\mathcal{U}_0,\quad \text{and}\quad V_0 =e^{tB^{-1}K}V_0+\int_0^te^{(t-s)B^{-1}K}B^{-1}\Gamma\mathcal{U}(s)\;ds.
$$
Therefore
\begin{equation}\label{eeqq1} 
U(t)-\begin{pmatrix}
e^{t\mathcal{A}_{T}}\mathcal{U}_0\\
0
\end{pmatrix}=\begin{pmatrix}
\mathbf{0}\\
e^{tB^{-1}K}{V}_0+\int_0^te^{(t-s)B^{-1}K}B^{-1}\Gamma\, \mathcal{U}(s)\;ds
\end{pmatrix}.
\end{equation}
Following the procedure outlined in Theorem \ref{nonExp}, we establish that \(\int_0^t U \, ds \in D(\mathcal{A})\). Using hypotheses \eqref{DomA} we get 
\[
\int_0^t M \, ds \in C([0,T]; H^2(0,\ell)), \quad M \in C([0,T]; L^2(0,\ell)).
\]
Consequently, by Theorem \ref{Kim} we have that \(\int_0^t M \, ds\) lies in a compact subset of \(C([0,T]; H^{2-\delta}(0,\ell))\). Thus, \(\int_0^t M_x \, ds\) lies in a compact subset of \(C([0,T]; H^{1-\delta}(0,\ell))\) for any initial data $U_0$ in a bounded set $B$ of $\mathcal{H}$. For \(0 < \delta < \frac{1}{2}\), this implies that \(\int_0^t M_x(0,s) \, ds\) lies in a compact subset of \(C([0,T])\). Since  \(e^{(t-s)B^{-1}K}\) and \(B^{-1}\) are bounded operators and recalling the definition of $\Gamma$ given in \eqref{yMatrC} we conclude that
\[
\mathfrak{G}(t) = \int_0^t e^{(t-s)K} B^{-1} \Gamma \mathcal{U}(s) \, ds
\]
lies in a compact set. Therefore, the operator
\[
\mathcal{T}(t) - \mathbf{T}(t)\Pi
\]
is compact, and the result follows.

\end{proof}

 \begin{theorem}\label{T00}
The semigroup $\mathcal{T}(t)=e^{\mathcal{A}t}$ is exponentially stable if and only if  $ \mathbf{T}(t)=e^{\mathcal{A}_Tt}$ is exponentially stable. \end{theorem}
\begin{proof}
Note that \(\mathcal{T}(t)\) and \(\mathbf{T}(t)\) are contraction semigroups. By Lemma \ref{T000} and Theorem \ref{TNovo}, to establish the exponential stability of \(\mathcal{T}(t)\) or \(\mathbf{T}(t)\), it suffices to show that the imaginary axis is contained in the resolvent set of their respective infinitesimal generators.

Suppose \(\mathcal{T}(t)\) is exponentially stable, but \(\mathbf{T}(t)\) is not. This implies that \(i\mathbb{R} \not\subset \varrho(\mathcal{A}_T)\), so there exists an eigenvector \(0 \neq W \in D(\mathcal{A}_T)\) and \(\lambda \in \mathbb{R} \setminus \{0\}\) such that
\[
\mathcal{A}_T W = i\lambda W.
\]
Define \(\widetilde{W} = (W, 0, 0)\). Then \(\widetilde{W}\in D(\mathcal{A})\) is an eigenvector of \(\mathcal{A}\), contradicting the exponential stability of \(\mathcal{T}(t)\).

Conversely, suppose \(\mathbf{T}(t)\) is exponentially stable, but \(\mathcal{T}(t)\) is not. This implies that there exists an eigenvector \(0 \neq \widetilde{W} \in D(\mathcal{A})\) and \(\lambda \in \mathbb{R} \setminus \{0\}\) such that
\[
\mathcal{A} \widetilde{W} = i\lambda \widetilde{W}.
\]
Taking the inner product with \(\widetilde{W}\) in \(\mathcal{H}\), we obtain
\[
i\lambda (\widetilde{W}, \widetilde{W})_{\mathcal{H}} - (\mathcal{A} \widetilde{W}, \widetilde{W})_{\mathcal{H}} = 0.
\]
By \eqref{dissipative}, for \(\widetilde{W} = (\widetilde{\mathcal{U}}, \widetilde{V})\), we have
\begin{equation}\label{dissx3}
\gamma |\widetilde{w}|^2 + \gamma_2 |\widetilde{z}|^2 \leq 0,
\end{equation}
which implies \(\widetilde{V} = 0\). Thus, \(W = \widetilde{\mathcal{U}}\) is an eigenvector of \(\mathcal{A}_T\), contradicting the exponential stability of \(\mathbf{T}(t)\). Hence, the result follows.
\end{proof}

\section{Applications}\label{sec:applications} 
\subsection*{The viscoelastic model}
The viscoelastic beam is characterized by $M=\alpha u_{xx}+\alpha_0 u_{xxt}$.
It is well know by now that the semigroup defined by equation \eqref{TTrge.1}--\eqref{bcnh} over the phase space 
$\mathbf{H}_0=\mathcal{V}\times L^1(0,\ell)$ is exponentially stable. 
The infinitesimal operator over $\mathbf{H}_0$ is given by 
$$
\mathcal{A}_T\mathcal{U}= \begin{pmatrix}
					v\\
					-\alpha u_{xxxx} -\alpha_0 v_{xxxx} 
					\end{pmatrix} ,\quad  \mathcal{U} = \begin{pmatrix}
u \\
v
\end{pmatrix}, \quad \forall\; \mathcal{U} \in \mathbf{H}.
					$$
It was showed in \cite{LL2} 	that the corresponding semigroup is exponentially stable, see also  \cite{[K1],[K2]}. For the hybrid problem  
the infinitesimal operator over  $\mathcal{H}= \mathcal{V}\times L^2(0,\ell) \times \mathbb{C}^2
$
is given by 
$$
\mathcal{A}U= \begin{pmatrix}
					v\\
					-\alpha u_{xxxx} -\alpha_0 v_{xxxx} \\
								B^{-1}\Gamma-B^{-1}KV
					\end{pmatrix} ,\quad   U = \begin{pmatrix}
u \\
v\\
V
\end{pmatrix}, \quad \forall\; {U} \in \mathcal{H}.
					$$
Thanks to  Theorem \ref{T00} we have that 
the  hybrid model \eqref{TTrge.1}, \eqref{GHib1}-\eqref{GHib2}, \eqref{bcnhh} is also exponentially stable.

\subsection*{The thermo elastic model of type I}
The thermo elastic beam model is characterized by $M=\alpha u_{xx}+m \theta $.
The corresponding model is given by 
	\begin{eqnarray}\label{Arge.HHH1}
 \rho u_{tt} + \alpha u_{xxxx}+m \theta_{xx}&=&0 \quad \text{in $]0,\ell[\times (0,+\infty)$,}\\
c\theta_t-\kappa \theta_{xx}+mu_{xxt}&=&0 \quad \text{in $]0,\ell[\times (0,+\infty)$,}\label{Arge.HHH2}
\end{eqnarray}
$$
u(\ell,t)=u_x(\ell,t)=0,\quad M(0,t)= M_x(0,t)=0,\quad \theta(0,t)=\theta(\ell,t)=0\quad t\geq 0
$$
with initial conditions 
$$
u(x,0)=u_0(x),\quad u_t(x,0)=u_1(x),\quad \theta(x,0)=\theta_0(x).
$$
the corresponding phase space is given by 
\begin{equation}\label{Phase2}
\mathbf{H}= \mathcal{V}\times L^2(0,\ell) \times L^2(0,\ell) 
\end{equation}
Hence the domain $\mathcal{D}(\mathcal{A}_{T})$ of the linear operator $\mathcal{A}_{T}:D(\mathcal{A}_{T})\subset \mathbf{H} \to \mathbf{H}$ is given by
\[
 \mathcal{D}(\mathcal{A}_{T})=\left\{
 \mathcal{U}\in \mathbf{H}: \;M \in H^2(0,\ell),\;
v \in  \mathcal{V},\; \theta\in H_0^1(0,\ell)\cap H^2(0,\ell)%,\;
\text{ verifying \;\;\eqref{bcnh}}
 \right\}.
\]
 $\mathcal{A}_{T}$ be the infinitesimal generator of  system \eqref{TTrge.1}--\eqref{bcnh} , \eqref{eqther}
	\begin{align}\label{Adef111}
				\mathcal{U}=\begin{pmatrix}
					u\\
					v\\
					\theta
					\end{pmatrix}, \quad \quad \mathcal{A}_T\mathcal{U}=\begin{pmatrix}
					v\\
					-\frac{1}{\rho}M_{xx}\\
					\frac{\kappa}{c}\theta_{xx}-\frac{m}{c}v_{xx}
					 \end{pmatrix} 
				\end{align}
				It was proved in \cite{[K3]} the exponential stability of the system, see also \cite{z3A895}. Thanks to  Theorem \ref{T00} we have that  
 the hybrid model \eqref{Arge.HHH1}--\eqref{Arge.HHH1}, \eqref{GHib1}-\eqref{GHib2}, \eqref{bcnhh} 
defined over the phase space is given by

\begin{equation}\label{PhaseH2}
\mathcal{H}= \mathcal{V}\times L^2(0,\ell) \times L^2(0,\ell)\times \mathbb{C}^2
\end{equation}
is also exponentially stable. 

\subsection*{The thermo elastic model of type II}

The thermoelastic model of type II is characterized by $M=\alpha u_{xx}+m \theta $ coupled with the temperature descreved  by equation 
\begin{equation}
c  \theta_{tt} - k^* \theta_{xx} - mu_{xxt} = 0, \quad x \in (0,L), \quad t > 0,
\label{eqther}
\end{equation}
Here the corresponding phase space is given by 
\begin{equation}\label{Phase3}
\mathbf{H}= \mathcal{V}\times L^2(0,\ell) \times H_0^1(0,\ell) \times L^2(0,\ell) 
\end{equation}
Denoting $v=u_t$ and $\Theta=\theta_t$, the corresponding infinitesimal operator is given by 
 $\mathcal{A}_{T}$ be the infinitesimal generator of  system \eqref{TTrge.1}--\eqref{bcnh} , \eqref{eqther}
	\begin{align}\label{Adef111}
				\mathcal{U}=\begin{pmatrix}
					u\\
					v\\
					\theta\\
					\Theta
					\end{pmatrix}, \quad \quad \mathcal{A}_T\mathcal{U}=\begin{pmatrix}
					v\\
					-\frac{1}{\rho}M_{xx}\\
					\Theta\\
					\frac{k^*}{c}\theta_{xx}+\frac{m}{c}v_{xx}
					 \end{pmatrix} 
				\end{align}
The domain $\mathcal{D}(\mathcal{A}_{T})$ of the linear operator $\mathcal{A}_{T}:D(\mathcal{A}_{T})\subset \mathbf{H} \to \mathbf{H}$ is given by
\[
 \mathcal{D}(\mathcal{A}_{T})=\left\{
 \mathcal{U}\in \mathbf{H}: \;M \in H^2(0,\ell),\;
v \in  \mathcal{V},\; \theta\in H^2(0,\ell),\; \Theta\in H_0^1(0,\ell)%,\;
\text{ verifying \;\;\eqref{bcnh}}
 \right\}.
\]
It is easy to see that $\mathcal{A}_{T}$ is the infinitesimal generator of a contraction semigroup. Note that  the model \eqref{Arge.HHH1}, \eqref{eqther} is conservative that is $\|e^{t\mathcal{A}_{T}}\|=1$. So the semigroup is not exponentially stable. 

As in the other cases the corresponding it is not difficult to see that operator defined in  \eqref{DefAT} is the infinitesimal generator of a contraction semigroup that defines the solutions of the hybrid model over the phase space

\begin{equation}\label{PhaseH3}
\mathcal{H}= \mathcal{V}\times L^2(0,\ell) \times H_0^1(0,\ell) \times L^2(0,\ell)\times \mathbb{C}^2
\end{equation}
Thanks to  Theorem \ref{T00} we have that  
 the hybrid model \eqref{Arge.HHH1}--\eqref{eqther}, \eqref{GHib1}-\eqref{GHib2}, \eqref{bcnhh} 
is not exponentially stable.

\subsection*{Non simple thermo elasticity }
We consider a beam composed by nonsimple
thermoelastic continua that occupies the interval $]0,\ell[$. The displacement and the temperature of typical particles are given by  $u(t,x)$ and $\theta(t,x)$, $x\in ]0,\ell[$. We denote by $T=\mu u_x+\theta$, $M=\alpha u_{xx}$ the stress and the hyper-stress. In the absent
of body forces the system of equations consists of  
$$
\rho u_{tt}=T_x-M_{xx},\quad \Xi_t= Q_x
$$
Where $\Xi=c\theta+m u_{xt}$ is the entropy density and $Q=\kappa \theta_x$ is the heat flux. 
In \cite {RQuintSare} 	was  proved that the model corresponding to nonsimple thermoelasticity given by equations

	\begin{eqnarray}\label{ATTrge.HHH1}
 \rho_1u_{tt} - \mu u_{xx}+ \alpha u_{xxxx}+m \theta_x&=&0 \quad \text{in $]0,\ell[\times (0,+\infty)$,}\\
c\theta_t-\kappa \theta_{xx}+mu_{xt}&=&0 \quad \text{in $]0,\ell[\times (0,+\infty)$,}\label{ATTrge.HHH2}
\end{eqnarray}
$$
u(\ell,t)=u_x(\ell,t)=u_{xx}(0,t)=u_{xxx}(0,t)=0,\quad \theta(0,t)=\theta(\ell,t)=0 \quad t\geq 0
$$
with initial conditions 
$$
u(x,0)=u_0(x),\quad u_t(x,0)=u_1(x),\quad \theta(x,0)=\theta_0(x).
$$
the corresponding phase space is given by 
\begin{equation}\label{Phase2}
\mathbf{H}= \mathcal{V}\times L^2(0,\ell) \times L^2(0,\ell) 
\end{equation}
Hence the domain $\mathcal{D}(\mathcal{A}_{T})$ of the linear operator $\mathcal{A}_{T}:D(\mathcal{A}_{T})\subset \mathbf{H} \to \mathbf{H}$ is given by
\[
 \mathcal{D}(\mathcal{A}_{T})=\left\{
 \mathcal{U}\in \mathbf{H}: \;M \in H^2(0,\ell),\;
v \in  \mathcal{V},\; \theta\in H_0^1(0,\ell)\cap H^2(0,\ell)%,\;
\text{ verifying \;\;\eqref{bcnh}}
 \right\}.
\]
 $\mathcal{A}_{T}$ be the infinitesimal generator of  system \eqref{TTrge.1}--\eqref{bcnh} , \eqref{eqther}
	\begin{align}\label{Adef111}
				\mathcal{U}=\begin{pmatrix}
					u\\
					v\\
					\theta
					\end{pmatrix}, \quad \quad \mathcal{A}_T\mathcal{U}=\begin{pmatrix}
					v\\
					-\frac{1}{\rho}T_x-\frac{1}{\rho}M_{xx}\\
					\frac{\kappa}{c}\theta_{xx}-\frac{m}{c}v_{x}
					 \end{pmatrix} 
				\end{align}
				It was proved in \cite{RQuintSare}  the exponential stability of the system. Thanks to  Theorem \ref{T00} we conclude that  
 the hybrid model \eqref{ATTrge.HHH1}--\eqref{ATTrge.HHH1}, \eqref{GHib1}-\eqref{GHib2}, \eqref{bcnhh} 
defined over the phase space  
$$
\mathcal{H}= \mathcal{V}\times L^2(0,\ell) \times L^2(0,\ell)\times \mathbb{C}^2
$$
is also exponentially stable.

				 \section*{Acknowledgments} The authors would like to thank to CNPq project 307947/2022-0 Brazil and Fondecyt Proyect 1230914 Chile for the financial support.

\section*{Conflict of interest}	
This work does not have any conflicts of interest.

	\end{document}